\documentclass{amsart}

\usepackage[pagewise]{lineno}

\usepackage{graphicx}
\usepackage{float}

\usepackage{amsmath}
\usepackage{amsthm}
\usepackage{amsfonts}
\usepackage{amssymb}
\usepackage{mathtools}

\usepackage{amsaddr}

\usepackage{subfiles}

\usepackage[dvipsnames]{xcolor}
\usepackage{color,soul}

\usepackage[square,numbers]{natbib}

\usepackage{enumerate}
\usepackage[shortlabels]{enumitem}

\theoremstyle{plain}
\newtheorem{thm}{Theorem}[section]
\newtheorem{prop}[thm]{Proposition}
\newtheorem{lem}[thm]{Lemma}

\newtheorem{conj}[thm]{Conjecture}

\theoremstyle{definition}
\newtheorem{defi}[thm]{Definition}

\allowdisplaybreaks

\title[Physical Measure on Hyperbolic Repelling Fixed Point]{Smooth Circle Covering with a Physical Measure on a Hyperbolic Repelling Fixed Point}

\author[1]{Rubio Gunawan}
\address[1]{Scuola Internazionale Superiore di Studi Avanzati (SISSA), Trieste, Italy}
\address[1]{Abdus Salam International Centre for Theoretical Physics (ICTP), Trieste, Italy.}
\thanks{We thank Stefano Luzzatto for his supervision, and the anonymous referee for the suggestion to improve the regularity of our results to $C^\infty$.}
\date{\today}

\begin{document}


\maketitle

\section{Main Result and Discussion}
\subsection{Setting and Literature Study}

In dynamical systems, given a point $x$, we often look for descriptions of long-term behavior of the forward orbit $\{x_n\}_{n\geq0}$. 
In ergodic theory we use the notion of empirical measure, which is the sequence of measures defined by $\frac{1}{n} \sum_{i = 0}^{n-1} \delta_{x_i}$.
Given a measure $\mu$, its basin of attraction $B_\mu$ is the set of points whose sequence of empirical measures weakly converge to $\mu$:
$B_\mu := \left\{x : \lim_{n \to \infty}\frac{1}{n} \sum_{i = 0}^{n-1} \delta_{x_i} \to \mu \right\}$.
A measure $\mu$ is called a physical measure if $B_\mu$ has positive Lebesgue measure.

A fixed point $p$ supporting a physical measure $\delta_p$ is an example of a statistical attractor.
A point $x \in B_{\delta_p}$ is not attracted to $p$ in the classical sense, but given an arbitrarily small neighborhood of $p$, for sufficiently large timescales the orbit will almost always be inside that neighborhood.
A topological attractor (such as a fixed point $p$ with $|f'(p)|<1$) is a trivial example of a statistical attractor.
A well-known example of statistical attractors that are not topological attractors are the intermittency maps of \cite{PM80} and \cite{LSV99}, which feature a topologically repelling indifferent fixed point $p$ with $f'(p) = 1$.

In this paper we will focus on the more interesting and counterintuitive situation of 
a hyperbolic repelling fixed point with $|f'(p)|>1$.
We coin a new term for fixed points that are hyperbolic repelling and also
statistical attractors.
\begin{defi}
    $p$ is called a Sisyphus Attractor if $|f'(p)| > 1$ and $|B_{\delta_p}|>0$.
\end{defi}

Hofbauer and Keller \cite{HofKel90} first proved the existence of Sisyphus Attractors for unimodal maps using the rich combinatorics of the kneading sequence within a full family.
We consider instead the setting of circle coverings, which can also be seen as full branch maps in which all branches are orientation preserving.

\begin{defi}
    $\mathcal{D}$ is the set of circle local homeomorphisms $f: \mathbb{S}^1 \to \mathbb{S}^1$ that are topologically conjugate to the doubling map ($x \mapsto 2x \text{ mod } 1$).
\end{defi}

As far as we know, there are no known examples of Sisyphus Attractors in the literature for maps in $\mathcal{D}$.
However, Benedicks-Misiurewicz \cite{BM89} and Zweim\"uller \cite{Zwe04} studied some unimodal maps with flat (all derivatives vanishing) critical points.
Both papers, using different arguments, show that integrability of $\log |f'|$ is equivalent to the existence of an absolutely continuous invariant probability measure (ACIP).
In particular, \cite{Zwe04} proves the existence of an induced Gibbs-Markov map with bounded distortion, and shows that existence of an ACIP is also equivalent to the integrability of first return time.
In the special case with full branches, by additional arguments, it should be possible to show that nonintegrability of first return time implies the existence of a Sisyphus Attractor.
By further arguments, including applying an involution, it might be possible to obtain a Sisyphus Attractor using the same methods for a circle covering, which might or might not be in~$\mathcal{D}$.

The purpose of this paper is to introduce a new method to find a Sisyphus Attractor for a circle covering map.
For $f \in \mathcal{D}$, let $p$ denote its unique fixed point, and $q$ denote the unique preimage of $p$ other than itself. We have:
\begin{thm}
\label{thm_smooth}
    $\exists f \in \mathcal D \cap C^\infty(\mathbb{S}^1)$ and $p$ is a Sisyphus Attractor.
\end{thm}
Our argument for the proof of Theorem \ref{thm_smooth} also uses a full branch induced map but, unlike \cite{Zwe04}, our induced map has \emph{unbounded distortion} and \emph{integrable return times}. It is therefore a completely different mechanism from those used in the previous literature for unimodal maps. 
We also emphasize that in order to have a $C^\infty$ construction, we needed $q$ to be a flat critical point, but unlike the case of the unimodal maps, this flatness is \emph{not} essential to obtain a Sisyphus Attractor. We will elaborate this role of flatness further in Section \ref{section_flat}.
In fact, we obtain a full measure $B_{\delta_p}$ by sacrificing smoothness and flatness at $q$:
\begin{thm}
\label{thm_full}
    $\exists f \in \mathcal D \cap C^\infty(\mathbb{S}^1 \setminus \{q\})$ and $p$ is a Sisyphus Attractor with $|B_{\delta_p}| = 1$.
\end{thm}
We also conjecture that $f \in C^\infty(\mathbb{S}^1)$ and $|B_{\delta_p}| = 1$
are compatible, but we were unable to prove it with our methods.
\begin{conj}
\label{conj}
     $\exists f \in \mathcal D \cap C^\infty(\mathbb{S}^1)$
     and $p$ is a Sisyphus Attractor with $|B_{\delta_p}| = 1$.
\end{conj}
Also observe that for the doubling map ($x \mapsto 2x \text{ mod } 1$), we have $|B_{\delta_p}| = 0$. 
This implies that the topological conjugacy of our results with the doubling map is singular.
The conjugacy also implies that the basins in our results are meagre.



\section{Outline of Proof}
\label{section_proof}
The general strategy for this paper is to find maps $f \in \mathcal{D}$ that have a particular first return map $F$,
in order to apply our previous results on wild attractors in \cite{Gun25}.

For the rest of the paper, for any given map $f \in \mathcal D$,
we will identify the circle $\mathbb{S}^1$ with the unit interval $I = [0,1]$ with endpoints identified,
the unique fixed point $p = 0 = 1$, and the other preimage of $p$ as a point $q \in (0,1)$.
We will then treat $f$ as an orientation preserving full branch map with 2 branches, with respect to 
the partition $\mathcal P = \{I_1 = [0,q), I_2 = [q,1)\}$.
We now study the first return map of $f$ on $I_2$.
First we introduce this related space of full branch maps defined on $[q,1]$.
\begin{defi}
Let $\mathcal{F}'$ denote the set of orientation-preserving full branch maps $F: [q,1] \to [q,1]$ , $F(q) = F(1) = 1$.
and branch domains $\{K_n\}_{n \in \mathbb N}$ defined by
\begin{equation}
\label{eqn_endpoints}
    K_n = [z_{n+1}, z_n),
\end{equation}
where $(z_n)_{n \geq 1}$ is a decreasing sequence with $z_1 = 1$ and $z_n \to q$.
\end{defi}
\begin{defi}
Let $\mathcal{F}$ be the set of maps in $\mathcal{F}'$ that have a generating partition.
\end{defi}
Let $f \in \mathcal D$, $\tau$ its first return time on $[q,1]$, and $F$ its first return map on $[q,1]$. Then $F \in \mathcal{F}$.
We now show how the dynamics of certain points by $F$ can give us information of their dynamics by $f$.
For a point $x \in I_2$, we have the sequence of return times $\{\tau_i(x)\}_{i \geq 0}$,
defined by $\tau_i(x) = \tau(F^{i}(x))$.
Define $E$ to be the set of points whose sequence of return times has a strictly increasing tail:
\begin{equation}
\label{eqn_E}
    E = \{x \in I_2 | \exists T \geq 0: \forall i \geq T, \tau_{i+1}(x) > \tau_i(x)\}.
\end{equation}
Note that $E$ is fully invariant by $F$. We then have:
\begin{prop}
\label{prop_C_in_basin}
    Let $f \in \mathcal{D}$, and $F$ be its first return map on $I_2$.
    Then $E \subset B_{\delta_0}$.
\end{prop}


We introduce some tools to estimate the Lebesgue measure of $E$.
Given $F \in \mathcal{F}$ defined on $I_2$, 
for each branch domain $K_n$, we define the following set:
\begin{equation}
    \label{eqn_K_n_minus}
    K_n^- := \bigcup_{m \geq n+1} K_m = (q, z_{n+1}).
\end{equation}
Note that $K_n^-$ is the union of the branch domains to the left of $K_n$. We also define:
\begin{align}
    \label{eqn_L_n}
    L_n &:= \left\{x \in K_n : F(x) \in K_n^-\right\} \subset K_n.
\end{align}
Note that $L_n$ is a subinterval of $K_n$ that shares its left endpoint.
We pay particular attention to maps $F$ such that $|L_n|/|K_n| \to 1$ fast enough:
\begin{defi}
\label{defi_F_star}
    $\mathcal{F}_*$ is the set of maps in $\mathcal{F}$ 
    such that there exists a monotonically increasing sequence $(p_n)$ of positive numbers such that:
    ${|L_n|}/{|K_n|} \geq p_{n}$ and $\prod_{n \geq 1} p_n > 0$.
    Furthermore we define two subfamilies:
    \begin{align*}
        \mathcal{F}_*^{weak} &:= \{F \in \mathcal{F} : \forall n, f|_{L_n}: L_n \to K_n^- \text{ is convex.} \}\\
        \mathcal{F}_*^{strong} &:= \{F \in \mathcal{F} : \forall n, f|_{K_n}: K_n \to I_2 \text{ is convex.} \}
    \end{align*}
\end{defi}
The subfamilies $\mathcal{F}_*^{weak}$ and $\mathcal{F}_*^{strong}$ give us estimates on $|C|$ and $|E|$:

\begin{prop}
\label{prop_physical_0}
    Let $f \in \mathcal D$.
    Let $F$ be its first return map on $[q,1]$.
    \begin{enumerate}
        \item If $F \in \mathcal{F}_*^{weak}$, then $|E|>0$. 
        In particular by Proposition \ref{prop_C_in_basin}, $|B_{\delta_0}| > 0$.
        \item If $F \in \mathcal{F}_*^{strong}$, then $|E| = |I_2|$. 
        In particular by Proposition \ref{prop_C_in_basin}, $|B_{\delta_0}| = 1$.
    \end{enumerate}
\end{prop}
By Proposition \ref{prop_physical_0}, the following 
Propositions \ref{prop_result_smooth} and \ref{prop_result_full} 
are now sufficient to prove our main Theorems \ref{thm_smooth} and \ref{thm_full} respectively.

\begin{prop}
\label{prop_result_smooth}
    There exists $f \in \mathcal{D} \cap C^\infty(\mathbb{S}^1)$ such that $f'(p) > 1$, and the first return map $F$ on $[q,1]$ is an element of $\mathcal{F}_*^{weak}$. 
\end{prop}

\begin{prop}
\label{prop_result_full}
    There exists $f \in \mathcal{D}$ such that $f'(p) > 1$, 
    $f \in C^\infty(\mathbb{S}^1 \setminus \{q\})$
    and the first return map $F$ on $[q,1]$ is an element of $\mathcal{F}_*^{strong}$
\end{prop}

A key ingredient for the proofs of Proposition \ref{prop_result_smooth} and \ref{prop_result_full} is
the realization method given by Proposition \ref{prop_finite_realization}, which shows how to find a map $f \in \mathcal {D}$
that has a given map $F \in \mathcal{F}'$ as its first return map.

We briefly outline the contents of the remaining sections.
In Section \ref{section_dynamics} we prove Proposition \ref{prop_C_in_basin}.
In Section \ref{section_dynamics_F} we prove Proposition \ref{prop_physical_0}.
In Section \ref{section_realization} we give a number of "realization-type" results which specify conditions on a pair $f_1:[0,q) \to [0,1)$ and $F \in \mathcal{F}'$ to guarantee the existence of $f_2$ which satisfies certain properties.
In Sections \ref{section_construction_smooth} and \ref{section_construction_full} 
we do the highly nontrivial construction of maps which satisfy the conditions of Section \ref{section_realization} to prove Propositions \ref{prop_result_smooth} and \ref{prop_result_full}.
In Section \ref{section_flat} we make remarks on the role of a flat critical point that appears in our results.

\section{Induced Dynamics}
\label{section_dynamics}
In this section we prove Proposition \ref{prop_C_in_basin}, which shows how the dynamics of the return time by the induced map can determine the empirical measure.
Fix a map $f \in \mathcal{D}$ and let $F \in \mathcal{F}$ be its first return map on $I_2$.
Recall the sequence of return times $\{\tau_i(x)\}_{i \geq 0}$, and the set of points with eventually strictly increasing return times $E$ as defined in (\ref{eqn_E}).
We define the sequence of cumulative return times.
$\forall i \geq 1, R_i(x) := \sum_{0 \leq j < i} \tau_j(x)$
For consistency, we define $R_0 := 0$.

\begin{proof}[Proof of Proposition \ref{prop_C_in_basin}] 
    Fix a point $x \in E$. 
    It is sufficient to prove that for any interval neighborhood $U$ of $0$, the proportion of visits to $U$ tends to 1, that is:
    \begin{align*}
        \lim_{J \to \infty} \frac{\#\{0\leq j < J: x_j \in U\}}{J} = 1.
    \end{align*}
    Now fix an interval neighborhood $U$ of $0$. Let $r$ be the right endpoint of $U$.
    Let $N$ be the smallest integer such that $f^N(r) \in I_2$.
    Because $f$ is an orientation preserving homeomorphism, points outside of $U$ must visit $I_2$ in time $N$ or less.
    Thus for the range of time between $R_i(x)$ and $R_{i+1}(x)$, the orbit of $x_j$ must spend at least $\tau_i(x) - N$ time inside $U$.
    Because $x \in E$, for $i$ sufficiently large, we have $\tau_i(x) >> N$.
    Therefore as $J \to \infty$, the proportion of times that $x_j$ spends in $U$ tends to $1$.
\end{proof}

\section{Dynamics of F}
\label{section_dynamics_F}
In this section we estimate the Lebesgue measure of $E$ to prove Proposition \ref{prop_physical_0}.
Let $f \in \mathcal{D}$, and $F \in \mathcal{F}$ be its first return map on $I_2$.
Recall the definition of $E$ in (\ref{eqn_E}) and the left subintervals $L_n \subset K_n$ in (\ref{eqn_L_n}).
Recall also that the partition $\mathcal{P}^F = \{K_n\}_{n \geq 1}$ is defined by first return time $\tau$ (i.e. $x \in K_n \iff \tau(x) = n$).
Therefore, we obtain the following expression for $E$ in terms of $L_n$:
\begin{equation*}
    E = \{x \in I_2| \exists T \geq 0: F^i(x) \in \bigcup_{m \geq 1} L_m \text{ for all } i \geq T\}.
\end{equation*}
Note that this new expression allows us to extend the definition of $E$ for general maps $F \in \mathcal{F}'$ without specifying that it is the first return map of some $f \in \mathcal {D}$.
Now we define a family of related subsets of $I_2$:
\begin{align*}
C &:= \{x \in I_2| F^i(x) \in \bigcup_{m \geq 1} L_m \text{ for all } i \geq 0\}.\\
C_n &:= \{x \in I_2| F^i(x) \in \bigcup_{m \geq 1} L_m \text{ for all } 
i \text{ such that } 0 \leq i < n\}.
\end{align*}
By definition, $\bigcap_{n \geq 1} C_n = C \subset E$.
With these definitions, Proposition \ref{prop_physical_0} is essentially proven in \cite{Gun25}.

\begin{proof}[Proof of Proposition \ref{prop_physical_0}]
    We prove the first item. Let $F \in \mathcal{F}_*^{weak}$. 
    By the proof of Proposition 2.1. of \cite{Gun25}, we have 
     $|E| > |C| \geq |[q,1]| P > 0$. 
    By Proposition \ref{prop_C_in_basin}, $|B_{\delta_0}| \geq |E| > 0$.
    
    Now we prove the second item. Let $F \in \mathcal{F}_*^{strong}$.
    By the proof of Proposition 2.2. of \cite{Gun25}, 
    we have $|E| = |[q,1]|$.
    Thus $E$ is a subset of $[q,1]$ with full Lebesgue Measure.
    Because $f \in \mathcal{D}$ and $E \subset B_{\delta_0}$, 
    we obtain $|B_{\delta_0}| = 1$.

    
\end{proof}

\section{Realization of Induced Maps}
\label{section_realization}
In this section we state several results, which show us how to find a 
two-branched map $f:[0,1) \to [0,1)$ that has certain properties
by choosing its first branch and first return map.
The following subsections will contain the proofs of these results.
For the rest of this section,
let $f_1 : [0,q) \to [0,1)$ be an orientation preserving nonsingular homeomorphism,
and $F \in \mathcal F'$ be defined on $[q,1]$. 
\begin{prop}[Realization Method]
\label{prop_finite_realization}
    There exists a unique homeomorphism $f_2 : [q,1) \to [0,1)$,
    such that the full branch map $f: [0,1) \to [0,1)$ defined by $(f_1,f_2)$ has first return map $F$ on $[q,1]$. 
\end{prop}
We give a sufficient condition such that $f \in \mathcal{D}$:
\begin{lem}
    \label{lem_generating_f_realized}
    If $F \in \mathcal F$ (which means $F$ has a generating partition), 
    the unique $f$ from Proposition \ref{prop_finite_realization} 
    has a generating partition and thus $f \in \mathcal{D}$.
\end{lem}

We give a sufficient condition on the pair $(f_1, F)$ such that $f_2$ is $C^\infty$ on $(q,1)$:
\begin{defi}[Adjusted First Return]
\label{defi_adjusted}
    A pair $(f_1,F)$ is adjusted if there exists 
    a sequence of positive real numbers $(m_n)_{n \geq 1}$ such that:
    \begin{enumerate}[label = (\roman*)]
        \item \label{item_left_adjusted} $\forall n\geq1$, there exists an interval
        $L_n = [z_{n+1}, z_{n+1} + \delta_n) \subset K_n$ 
        where $F$ is affine with slope $m_n$,
        which means $F|_{L_n}(x) = q + m_n (x-z_{n+1})$.
        \item \label{item_right_adjusted} $\forall n\geq2$, there exists an interval
        $R_n =  (z_n - \varepsilon_n, z_n) \subset K_n$
        where the graph of $F$ is the graph of $f_1$ around $q$, contracted horizontally by a factor of $m_{n-1}$,
        which means $F|_{R_n}(x) = f_1( q + m_{n-1}(x - z_n))$.
        \item \label{item_smooth_f1_adjusted} $f_1$ is a $C^\infty$ diffeomorphism.
        \item \label{item_smooth_adjusted} $F$ has $C^\infty$ branches.
    \end{enumerate}
\end{defi}

\begin{lem}[Adjusted Gluing Lemma]
\label{lem_adjusted_gluing}
    Let the pair $(f_1,F)$ be adjusted.
    Take the unique $f_2: [q,1) \to [q,1)$ given by Proposition \ref{prop_finite_realization}.
    Then $f_2$ is $C^\infty$ on $(q,1)$.
\end{lem}

We give a sufficient condition on the pair $(f_1,F)$ to glue $f_1,f_2$ smoothly at $q$:

\begin{defi}[Gluing Flatly]
    \label{defi_glueflat}
    An adjusted pair $(f_1,F)$ glues flatly if it satisfies:
    \begin{enumerate}[label=(\roman*)]
        \item \label{item_flat_f1} $f_1$ extends as a $C^\infty$ diffeomorphism on $[0,q]$, such that $q$ is a flat critical point (i.e. $\forall n \geq 1, f^{(n)}(q) = 0$).
        \item \label{item_flat_m} There exists $b \in [0,q)$ 
        such that $f_1|_{[0,b]} : [0,b] \to [0,q]$ is affine with slope $m = q/b > 0$.
        Thus the backward orbit of $q$ satisfies
        $f_1^{-n}(q) = q m^{-n}$.
        \item \label{item_flat_endpoints} 
        There exists $\alpha > 0$ such that the sequence of endpoints $z_n$ (as defined in (\ref{eqn_endpoints})) satisfies
        $z_n - q \approx n^{-\alpha}$,
        where the notation $\approx$ means the ratio of the two expressions are bounded above and below by uniform positive constants.
        \item \label{item_flat_Q} $\forall k \geq 0$, there exists $C_k, \alpha_k > 0$ such that:
        $
            \max_{x \in K_n} |F^{(k)}(x)| \leq C_k n^{\alpha_k}.
        $
        \item \label{item_flat_P} For any degree $k \geq 1$, there exists $D_k > 0$ and $\beta_k \in (0, 1/\alpha)$ such that
        $
            \max_{x \in K_n} |(f_1^{-1})^{(k)} (F(x))| \leq \exp(D_k n^{\beta_k}).
        $
    \end{enumerate}
\end{defi}

\begin{lem}[Flat Gluing Lemma]
    \label{lem_glueflat}
    Let $(f_1,F)$ glue flatly.
    Then the full branch map $f$ given by Proposition \ref{prop_finite_realization}
    is $C^\infty$ at $q$, and $q$ is a flat critical point.
\end{lem}

\subsection{Realization}
\begin{proof}[Proof of Proposition \ref{prop_finite_realization}]
    We will explicitly construct an $f_2$ which proves Proposition \ref{prop_finite_realization}.
    We first define a new family of subintervals:
    \begin{align}
    \label{eqn_define_Jn}
        \forall n \geq 1, J_n := f_1^{-(n-1)}([q,1)).
    \end{align}
    Essentially, $J_n$ is the set of points in $I$ such that its first visit to $[q,1)$ is in time $(n-1)$.
    In particular, $J_1 = [q,1)$.
    Note that the family $\{J_n\}_{n \in \mathbb N}$ partitions $(0,1)$.
    The subintervals $J_n$ accumulate (monotonically in $n$) only at the left endpoint $0$, 
    in a similar manner to the subintervals $K_n$ of $\mathcal{P}^F$.
    Recall we want $f_2$ to be defined such that $F$ is the first return map of $f$ on $[q,1]$.
    Fix a partition element $K_n$, which we want to have points with return time $n$.
    $f_2|_{K_n}$ must satisfy $(f_1)^{n-1} \circ f_2|_{K_n} = F|_{K_n}$.
    So we must define $f_2|_{K_n}$ as follows:
    \begin{align}
    \label{eqn_f_pieces}
        \forall n \geq 1, f_2|_{K_n} := (f_1)^{-(n-1)} \circ F|_{K_n}.
    \end{align}
    We now know each piece $f_2|_{K_n} : K_n \to J_n$ is a nonsingular orientation-preserving homeomorphism.
    Because $\{K_n\}_{n \in \mathbb N}$ partitions $(q,1)$ 
    and $\{J_n\}_{n \in \mathbb N}$ partitions $(0,1)$ with the same accumulating behavior,
    by gluing and defining $f_2(q) = 0$, 
    we obtain a unique nonsingular orientation-preserving 
    homeomorphism $f_2 : [q,1) \to [0,1)$.
    We also obtain a full branch map $f$, and
    $f$ has $F$ as its first return map on $[q,1)$. 
\end{proof}

\subsection{Generating Partition}
\label{subsection_GP}
Lemma \ref{lem_generating_f_realized} is an immediate consequence of the following general lemma
on generating partitions of first return maps:
\begin{lem}
    \label{lem_generating_f}
    Let $f$ be a full branch map with respect to a countable partition $\mathcal{P} = \{I_n\}$.
    Let $\Delta$ be a union of elements of $\mathcal{P}$.
    Let $F: \Delta \to \Delta$ be the first return map, which is full branch with respect to a countable partition $\mathcal{P}^F = \{J_m\}$.
    If $F$ has a generating partition, then $f$ has a generating partition.
\end{lem}
\begin{proof}
    Take a pair of (non-exceptional) points $x_1,y_1 \in I_n \in \mathcal{P}$.
    Pick $I_m \subset \Delta$, and let $x = f^{-1}|_{I_m} (x_1)$, $y = f^{-1}|_{I_m} (y_1)$. 
    Thus $x,y \in \Delta$.
    Because $F$ has a generating partition, 
    $\exists M \geq 0$ such that $x,y$ lie on different elements of the refinement $\mathcal{P}_M^F$.
    Let $J$ be the partition element of $\mathcal{P}_M^F$ that contains $x$ (and does not contain $y$). 
    Observe that $F^M:J \to \Delta$ is a homeomorphism. 
    By the definition of $F$, $\exists N \geq 0$ such that
    $f^N: J \to \Delta$ is a homeomorphism.
    Recall $\Delta$ is a union of elements of $\mathcal{P}$,
    thus the refinement $\mathcal{P}_{N+1}$ has a collection of elements that partition $J$.
    Therefore, $x$ and $y$ must also lie on different elements of $\mathcal{P}_{N+1}$.
\end{proof}

\subsection{Gluing in $(0,q)$}
\label{subsection_pieces}
\label{section_gluing}

To prove Lemma \ref{lem_adjusted_gluing}, we study the regularity of $f_2$ in $(0,q)$.
For this purpose, we introduce the notation of graph pieces.
Given a pair of integers $\ell \geq m \geq 1$, the graph piece $\phi_{\ell,m}: K_\ell \to J_m$
is an orientation preserving homeomorphism defined by the following composition:
\begin{equation}
\label{eqn_define_piece}
    \phi_{\ell,m} := f_1^{\ell-m} \circ f_2|_{K_\ell}.
\end{equation}

\begin{figure}[H]
    \centering
    \includegraphics[width=0.4\linewidth]{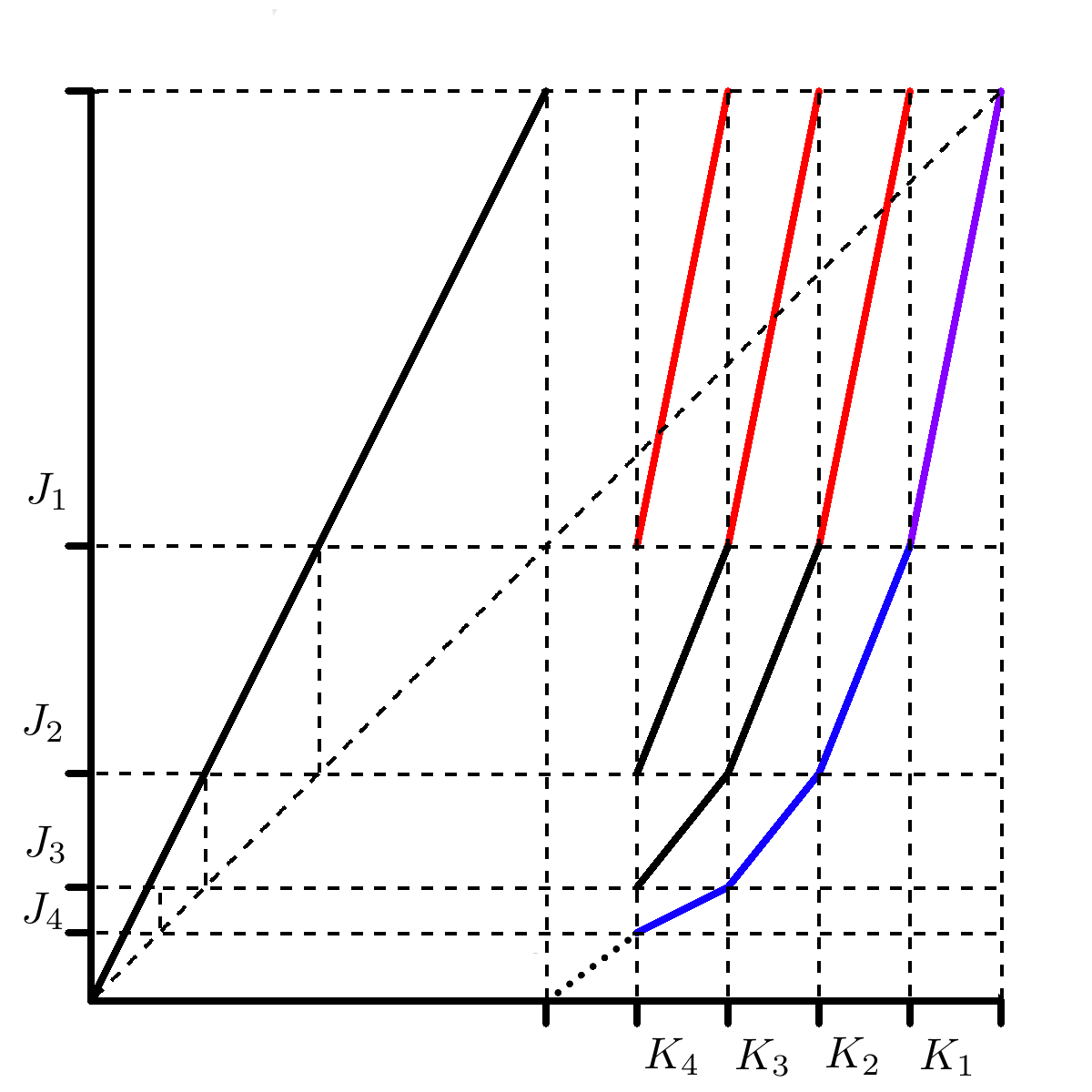}
    \caption{Example of $f$ and $\phi_{\ell,m}$ with $\ell \leq 4$.}
\end{figure}

We can also express the pieces of $f_2$ and the branches of $F$  as graph pieces:
\begin{align}
    \label{eqn_piece_ID} \forall n \geq 1: 
    & \phi_{n,n} = f_1^{n-n} \circ f_2|_{K_n} =  f_2|_{K_n},
    & \phi_{n,1} = f_1^{n-1} \circ f_2|_{K_n} = F|_{K_n}.
\end{align}
We make an elementary observation on the regularity of $f_2$.
Suppose $f_1$ and $F$ have a certain regularity (say $C^r, r \geq 1$).
Then by (\ref{eqn_f_pieces}) and by composition, $f_2$ has the same regularity on each interior $\mathring K_n$.
This means to prove the regularity of $f_2$, we only need to prove it for the gluing of consecutive pieces.
In fact, we have the following general lemma for $C^\infty$ regularity:
\begin{lem}[General Gluing Lemma]
\label{lem_gluing}
    Let $f_1 : [0,q) \to [0,1)$ be a $C^\infty$ diffeomorphism,
    let $F \in \mathcal{F}'$ have $C^\infty$ branches,
    and let $f_2$ be the unique second branch given by Proposition \ref{prop_finite_realization}.
    Then the following three statements are equivalent:
    \begin{enumerate}[label=(\roman*)]
        \item $f_2: [q,1) \to [q,1)$ is $C^\infty$ on $(q,1)$.
        \item $\forall n \geq 1$, $\phi_{n,n}$ and $\phi_{n+1,n+1}$ glue in a $C^\infty$ way at a neighborhood of $z_{n+1}$.
        \item $\forall n \geq 1$, $\phi_{n,1}$ and $\phi_{n+1,2}$ glue in a $C^\infty$ way at a neighborhood of $z_{n+1}$.
    \end{enumerate}
\end{lem}
\begin{proof}
    Recall by (\ref{eqn_piece_ID}) that $f_2$ is the gluing of pieces $\{\phi_{n,n}\}_{n \geq 1}$.
    By construction, $f_2$ has $C^\infty$ regularity on each interior $\mathring K_n$, thus (i) and (ii) are equivalent.
    Now fix an index $n$. We have the following equalities:
    \begin{align*}
        f_1^{n-1} \circ \phi_{n,n} &= \phi_{n,1},
        &f_1^{n-1} \circ \phi_{n+1,n+1} &= \phi_{n+1,2}.
    \end{align*}
    Because $f_1^{n-1}$ is $C^\infty$, the pair $(\phi_{n,n}, \phi_{n+1,n+1})$ glues in a $C^\infty$ way if and only if the pair $(\phi_{n,1}, \phi_{n+1,2})$ glues in a $C^\infty$ way.
    Thus (ii) and (iii) are equivalent.
\end{proof}
\begin{proof}[Proof of Lemma \ref{lem_adjusted_gluing}]
    We want item (i) of Lemma \ref{lem_gluing}, so it is sufficient to verify item (iii).
    Fix $n \geq 1$, and take $U = L_n \bigcup R_{n+1}$ as the neighborhood of $z_{n+1}$.
    Because the pair is adjusted, we obtain:
    \begin{align*}
        \phi_{n,1}|_U(x) = F|_{L_n}(x) &= q + m_n (x-z_{n+1}),\\
        \phi_{n+1,2}|_U(x) = f_1^{-1} \circ F|_{R_{n+1}}(x) = f_1^{-1}(f( q + m_{n}(x - z_{n+1}))) &= q + m_n (x-z_{n+1}).
    \end{align*}
    The two pieces together define an affine map $x \mapsto q + m_n (x-z_{n+1})$ on $U$, thus they glue smoothly at $q$.
\end{proof}

\subsection{Flat gluing at $q$}
\label{subsection_flat}
Lemma \ref{lem_adjusted_gluing} gives a sufficient condition for $f$ to be $C^\infty$ in $(q,1)$.
However, this does not imply a $C^\infty$ extension to $[q,1)$, because the derivatives might oscillate close to $q$.
In particular if $F \in \mathcal{F}_*$, we should expect derivatives of $f_2$ to oscillate greatly.
To prove Lemma \ref{lem_glueflat}, we must bound this oscillation.
\begin{proof}[Proof of Lemma \ref{lem_glueflat}]
    We claim the following limit holds for any $k \geq 0$:
    \begin{equation}
    \label{eqn_flat_asymptote}
        \limsup_{x \to q^+} |f_2^{(k)}(x)|/ \exp(-(x-q)^{- 1/\alpha }) < \infty.
    \end{equation}
    Observe $\exp(-(x-q)^{-1/\alpha})$ 
    extends to a $C^\infty$ diffeomorhism on $[q,1)$ 
    such that $q$ is a flat critical point.
    By (\ref{eqn_flat_asymptote}) and Squeeze Theorem, we obtain $f_2$ also extends to a 
    $C^\infty$ diffeomorphism on $[q,1)$ such that $q$ is a flat critical point.
    Thus by item \ref{item_flat_f1} of Definition \ref{defi_glueflat}, 
    $f$ is $C^\infty$ at $q$, and $q$ is a flat critical point.
    
    We prove (\ref{eqn_flat_asymptote}) for $k = 0$.
    Take $x \in [q,1)$.
    There is a unique $n > 1$ such that $x \in [z_{n+1},z_n)$.
    By \ref{item_flat_endpoints}, $n \approx (x - q)^{-1/\alpha}$.
    By \ref{item_flat_m}, $\exists C > 0$ such that:
    \begin{align}
        \label{eqn_flat_asymptote_0}
        &f_2|_{K_n}(x) \leq f_1^{-(n-1)}(1) = f_1^{-(n-2)}(q) = q m^{-(n-2)}\\
        \nonumber &\, = \exp(- \log (m) (n-2)) \leq \exp(- C (x-q)^{-1/\alpha}).
    \end{align}
    
    Now we prove (\ref{eqn_flat_asymptote}) for $k \geq 1$. 
    Recall by (\ref{eqn_f_pieces}), $f_2|_{K_n} = f_1^{n-1} \circ F|_{K_n}$.
    By \ref{item_flat_m}, 
    \begin{equation}
    \label{eqn_flat_decompose}
        f_2|_{K_n} = m^{-(n-2)} (f_1^{-1} \circ F|_{K_n}).
    \end{equation}
    We want to bound the $k$-th derivative of $f_1^{-1} \circ F|_{K_n}$ evaluated at $x \in K_n$.
    By Fa\`a di Bruno's formula, the $k$-th derivative of a general composition $h \circ g$ evaluated at $x$
    is a polynomial $P_k$ with $2k$ variables
    evaluated at $h^{(m)}(g(x))$ and $g^{(m)}(x)$ with $1 \leq m \leq k$.
    Recall the bounds of \ref{item_flat_Q}, \ref{item_flat_P}.
    Let $\gamma_k = \max_{m \leq k} \beta_k$. Note $\gamma_k < 1/\alpha$.
    As $n \to \infty$, the fastest growing bound is $ \exp(D_k n^{\gamma_k})$.
    Let $S_k$ be the sum of absolute value of coefficients of $P_k$. 
    For sufficiently large $n$, $\exists C_k > 0$ such that:
    \begin{align}
    \label{eqn_flat_Faa}
    |(f_1 \circ F|_{K_n})^{(k)}(x)| \leq S_k [\exp(D_k n^{\gamma_k})]^{\deg(P_k)}
    \leq \exp(C_k (x-q)^{-\gamma_k}).
    \end{align}
    By (\ref{eqn_flat_asymptote_0}), $m^{-(n-2)} \leq \exp(-C(x-q)^{-1/\alpha} - \log q)$.
    Substituting together with (\ref{eqn_flat_Faa}) to (\ref{eqn_flat_decompose}), we obtain:
    $|f_2^{(k)}(x)| \leq \exp(C_k (x-q)^{-\gamma_k} -C(x-q)^{-1/\alpha} - \log q)$
    for $x$ sufficiently close to $q$.
    Because $\gamma_k < 1/\alpha$,
    we have (\ref{eqn_flat_asymptote}) for $k \geq 1$.

\end{proof}

\section{Smooth Construction} 
\label{section_construction_smooth}
In this section we prove Proposition \ref{prop_result_smooth}.
Our strategy is to choose a suitable pair $(f_1,F)$, 
then we obtain a unique $f$ by Proposition \ref{prop_finite_realization}.
First, fix a point $q \in (0,1/2)$, and define $\mathcal{P} := \{I_1 = [0,q), I_2 = [q,1)\}$.

\subsection{Choice of $f_1$}
\label{subsection_f1_smooth}
We choose $f_1:[0,q) \to [0,1)$ to be a $C^\infty$ orientation preserving homeomorphism that satisfies three properties:
\begin{itemize}
    \item There exists a point $b \in (0,q)$, such that the restriction $f_1|_{[0,b]}:[0,b] \to [0,q]$
    is affine and has slope:
    \begin{equation}
        \label{eqn_m}
        m = q/b > 1.
    \end{equation}
    \item Define $\varphi: [0,1] \to [0,e^{-9}]$ by $\varphi(0) := 0$, $\varphi(x) := \exp (-9x^{(-1/9)})$.
    There exists a one-sided open neighborhood $U$ of $q$ such that:
    \begin{equation}
        \label{eqn_varphi_ngbd}
        f_1|_{U} (x) := 1 - \varphi(q-x).
    \end{equation}
    Note because $0$ is a flat critical point of $\varphi$, we have
    \begin{equation}
        \label{eqn_critical_q}
        \forall k \geq 1, f_1^{(j)}(q) = 0,
    \end{equation}
    i.e. $q$ is a flat critical point of $f_1$.
    \item We also assume that 
    \begin{equation}
        \label{eqn_f_1_nonflat}
        \forall x \in [0,q] \setminus U : f'(x) > 0.
    \end{equation}
    Recall the function $x \to 1/x$ is smooth away from $0$, and $(f^{-1})'(x) = 1/f'(x)$.
    Thus $f^{-1}$ is also $C^\infty$ on the compact set $[0,q] \setminus U$,
    and as a consequence the derivatives of $f^{-1}$ 
\end{itemize}
For future reference, we also note bounds of $\varphi(x)$ and derivatives of $\varphi^{-1}(y)$.
First:
\begin{equation}
    \label{eqn_varphi_bound}
    \varphi(x) < x.
\end{equation}
Furthermore, we have $\varphi^{-1}(y) = (-(\log y) / 9)^{-9}$.
Note  $(\varphi^{-1})'(y) = y^{-1} (-(\log y) / 9)^{-10}$.
By induction, we have for a sequence of polynomials $(\phi_k)_{k \geq 1}$:
\begin{equation}
    \label{eqn_varphi_inv_bounds_y}
    (\varphi^{-1})^{(k)}(y) = y^{-k} \phi_k\left((-(\log y) / 9)^{-1}\right).
\end{equation}
Substituting $y = \varphi(x)$, we obtain for a constant $C_k$ depending on $k$:
\begin{equation}
    \label{eqn_varphi_inv_bounds}
    (\varphi^{-1})^{(k)}(\varphi(x)) = \varphi(x)^{-k} \phi_k (x^{1/9}) \leq C_k \varphi(x)^{-k} 
\end{equation}

\subsection{Choice of Partition $\mathcal{P}^F$}
\label{subsection_smooth_partition}
Before we choose $F$, we must choose a partition $\mathcal{P}^F= \{K_n\}_{n \geq 1}$.
This is done by choosing their endpoints $z_n$ (as in (\ref{eqn_endpoints})):
\begin{align}
    \label{eqn_endpoint_smooth}
    z_1:= 1,\quad z_2 := 1-(1-q)/m,\quad
    \forall n \geq 3, z_n := q + (z_2 - q) 2 [(n-1)n]^{-1}.
\end{align}
Note also $z_2 = q + (z_2 - q) 2 [(2-1)2]^{-1}$. We note their lengths:
\begin{align}
    \label{eqn_length_smooth}
    |K_1| &= z_2 - z_1 = (1-q)/m,\\
    \nonumber \forall n \geq 2, |K_n| &= z_n - z_{n+1} = (z_2 - q) 4 [(n-1)n(n+1)]^{-1}.
\end{align}
Recall the union $K_n^- = \bigcup_{k \geq n+1} K_n$ defined in (\ref{eqn_K_n_minus}).
We note its length.
\begin{align}
    \label{eqn_union_length_smooth}
    |K_n^-| &= z_{n+1} - q = (z_2 - q)[n(n+1)]^{-1}.
\end{align}

\subsection{Construction of $F$}
The following Lemma \ref{lem_F_smooth} specifies a choice of $F \in \mathcal{F}'$ with several desired properties for further reference.

\begin{lem}
\label{lem_F_smooth}
There exists a map $F \in \mathcal{F}'$ with branch domains given by $\mathcal{P}^F$
such that $(f_1,F)$ is adjusted, and together with the $L_n,R_n$ given in Definition \ref{defi_adjusted} satisfy:
\begin{enumerate}[label= (\roman*)]
    \item \label{item_affine_smooth} The first branch $F|_{K_1}$ is affine with slope $m_1 = m$.
    \item \label{item_ratio_smooth} $\prod_{n \geq 1} |L_n| / |K_n| > 0$.
    \item \label{item_image_smooth} $\forall n \geq 1, F(L_n) = K_n^- = \bigcup_{k \geq n+1} K_n$
    \item \label{item_slope_smooth} For $n \geq 2$, $F|_{L_n}$ is affine with constant slope $m_n$, and $n-1 < m_n < n$.
    \item \label{item_Rn_smooth} $|R_n| \approx n^{-5}$.
    \item \label{item_varphi_smooth} $\forall x \geq 2$, $F|_{R_n}(x) = 1 - \varphi(m_{n-1} (z_n - x))$.
    \item \label{item_GP_smooth} $F|_{K_n}'(x) \geq m_n$, except on a subinterval of the form $[z_n - d_n, z_n)$ 
    where $F$ is concave, and $d_n < |K_1|/2$.
    \item \label{item_derivatives_smooth} $\forall k \geq 1$, 
    there exists positive numbers $C_k, \alpha_k$ 
    such that $F_{K_n}^{(k)}(x) \leq C_k n^{\alpha_k}$.
\end{enumerate}
\end{lem}
\begin{proof}
    To satisfy \ref{item_affine_smooth}, let $F|_{K_1} : K_1 \to [q,1)$ to be affine. 
    By (\ref{eqn_length_smooth}), it has constant slope $(1-q)/K_1 = m$.
    To satisfy item (i) of Definition \ref{defi_adjusted} and \ref{item_image_smooth} for $n = 1$, let $L_1$ be the subset of $K_1$ that maps to $K_1^-$.
    We now define the other branches.
    Let $N$ be a large integer to be fixed later.
    Now we subdivide $K_n$ into three subintervals:
    \begin{align}
        \label{eqn_smooth_Ln} L_n := &[z_{n+1}, z_n - 2 (N+n)^{-2}|K_n|), 
        &|L_n| = (1-2 (N+n)^{-2})|K_n| \\
        \label{eqn_smooth_Rn}
        R_n := &[z_n - (N+n)^{-2} |K_n|, z_n), &|R_n| = (N+n)^{-2}|K_n|,\\
        M_n := &K_n \setminus (L_n \cup R_n), &|M_n| = (N+n)^{-2}|K_n| = |R_n|.
    \end{align}
    Observe that $|L_n|/|K_n| = 1- 2 (N+n)^{-2}$. 
    Because $\sum 2 (N+n)^{-2} < \infty$ and $\log(1 + x) \approx x$ for small $x$, we conclude \ref{item_ratio_smooth}.
    To satisfy item (i) of Definition \ref{defi_adjusted} and \ref{item_image_smooth} for $n \geq 2$, 
    let $F|_{L_n} : L_n \to K_n^-$ to be affine.
    By (\ref{eqn_length_smooth}), (\ref{eqn_union_length_smooth}), and (\ref{eqn_smooth_Ln}):
    \begin{equation}
        \label{eqn_smooth_mn}
        m_n = |K_n^-|/|L_n| = (1- 2 (N+n)^{-2})^{-1} (n-1).
    \end{equation}
    Note because $1 < (1- 2 (N+n)^{-2})^{-1} < (1 - n^{-1})^{-1}$, we have $(n-1) < m_n < n$, and thus \ref{item_slope_smooth}.
    To confirm \ref{item_Rn_smooth}, we compute by (\ref{eqn_length_smooth}) and (\ref{eqn_smooth_Rn}):
    \begin{equation}
        |R_n| = (N+n)^{-2}|K_n| =  (z_2 - q)4[(n-1)n(n+1)]^{-1} (N+n)^{-2} \approx n^{-5}.
    \end{equation}
    Because $F$ is adjusted, we must have $F|_{R_n}(x) = f_1(q + m_{n-1}(x- z_n))$.
    By (\ref{eqn_varphi_ngbd}), on a onesided neighborhood $U$ of $q$,
    we have $f_1(x) = 1 - \varphi(q - x)$.
    Assume $N$ is large enough such that $n|M_n \cup R_n| < |U|$ for all $n$.
    By \ref{item_slope_smooth}, $m_n|R_n| < |U|$.
    To satisfy item (ii) of Definition \ref{defi_adjusted}:
    \begin{equation}
        F|_{R_n}(x) = 1 - \varphi (m_{n-1}(x- z_n)),
    \end{equation}
    which verifies \ref{item_varphi_smooth}.
    Now we define a nonsmooth $F|_{M_n}$, to smoothen later.
    We subdivide $M_n$ into three subintervals of length $|M_n|/3$.
    On the left subinterval, $F(x) = q + m_n (x - z_{n+1})$, so it glues to $F|_{L_n}$ smoothly.
    On the right subinterval, $F(x) = 1 - \varphi (m_{n-1}(x- z_n))$, so it glues to $F|_{R_n}$ smoothly.
    On the middle subinterval, we define $F$ by an affine line segment.
    To smoothen $F|_{M_n}$, let $g$ be a nonnegative $C^\infty$ bump function supported on $[0,1]$.
    Let $g_n := g(x/(|M_n|/3))$, it is compactly supported on $[0, |M_n|/3]$.
    We obtain a smooth $F|_{M_n}$ by convolution of $g_n$ against the nonsmooth $F$.
    The three pieces $F|_{L_n}, F|_{M_n}, F|_{R_n}$ are smooth glue smoothly, which satisfies \ref{item_smooth_adjusted} of Definition \ref{defi_adjusted}.
    Furthermore because convolution with positive kernels preserves convexity and concavity, \ref{item_GP_smooth} holds for $d_n = |R_n| + 2|M_n|/3$, and we can assume $N$ is large enough such that $d_n < |K_1|/2$ for all $n$.
    Fix $k\geq 1$.
    Observe $\max |g_{n}^{(k)}| = (|M_n|/3)^{-k} \max |g^{(k)}(x)| \approx (n^{-5})^{-k}$.
    By differential properties of convolution, we obtain \ref{item_derivatives_smooth}.
    Furthermore the three pieces $F|_{L_n}, F|_{M_n}, F|_{R_n}$ glue smoothly, which satisfies \ref{item_smooth_adjusted} of Definition \ref{defi_adjusted}.
\end{proof}

\begin{figure}[H]
    \centering
    \includegraphics[width=0.35\linewidth]{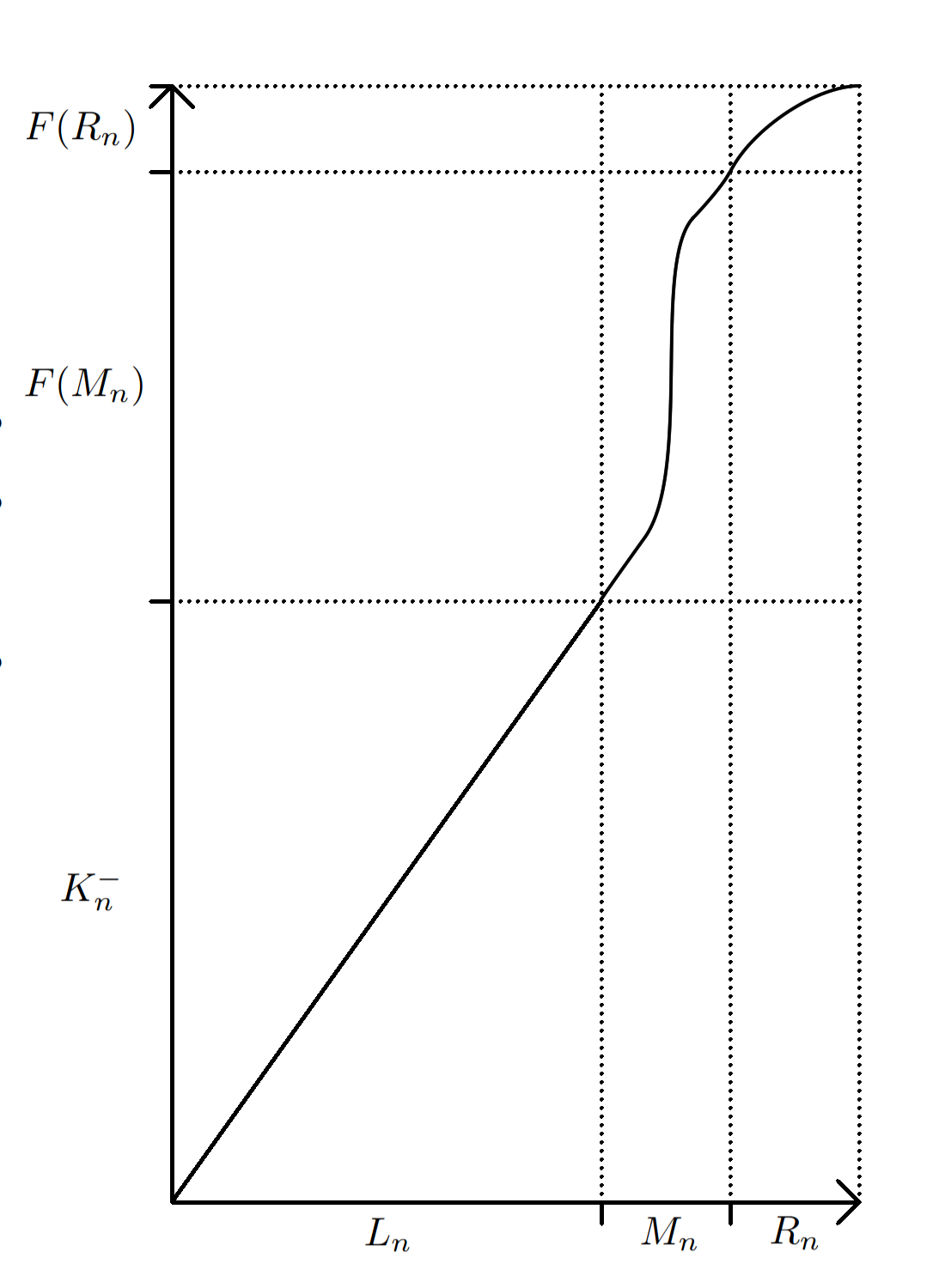}
    \caption{Sketch of the three pieces of $F|_{K_n}$}
\end{figure}

\subsection{Verification of Properties}
\label{subsection_full_verify}
We have chosen $f_1$ in Subsection \ref{subsection_f1_smooth}, 
and $F \in \mathcal{F}'$ in Lemma \ref{lem_F_smooth}.
By Proposition \ref{prop_finite_realization}, we obtain a full branch map $f: I \to I$ 
with branches defined on $\mathcal{P}=\{[0,q),[q,1)\}$.
To prove Proposition \ref{prop_result_smooth}, we must verify four properties:
$f \in \mathcal{D}$, $F \in \mathcal{F}_*^{weak}$, $f \in C^\infty(\mathbb{S}^1)$,
and $f'(p) > 1$.
\begin{lem}
    \label{lem_GP_smooth}
    $F$ has a generating partition, thus
    $F \in \mathcal{F}$. By Lemma \ref{lem_generating_f}, $f \in \mathcal{D}$.
\end{lem}
\begin{proof}
    Consider $\Delta = \bigcup_{n \geq 2}K_n$.
    Let $G : \Delta \to \Delta$ be the first return map of $F$.
    Again by Lemma \ref{lem_generating_f}, it is sufficient to show that
    $G$ is uniformly expanding and thus has a generating partition.
    Points in the dynamic of $F: [q,1) \to [q,1)$ outside of $\Delta$ must spend time in $K_1$, and by item \ref{item_affine_smooth} of Lemma \ref{lem_F_smooth}, $F|_{K_1}'(x) = m > 1$.
    Furthermore by chain rule,
    $G'(x) = (F^{\tau(x)})'(x) = m^{\tau(x) - 1} F'(x)$.
    Now consider $n \geq 2$, we study $G'$ and $F'$ at $K_n$.
    Recall item \ref{item_GP_smooth} and \ref{item_slope_smooth} of Lemma \ref{lem_F_smooth}.
    Outside of a subinterval of the form $[z_n -d_n,z_n)$, we have:
    \[G'(x) \geq F'(x) \geq m_n > n-1 \geq 1.\]
    Now we take a point $z \in [z_n -d_n,z_n)$.
    $(F^{\tau(z)})$ is a concave $C^\infty$ diffeomorphism on $[z_n -d_n, z_n)$.
    It extends continuously to $[z_n -d_n, z_n]$ by $z_n \mapsto 1$.
    By the definition of return time, $F^{\tau(z)} \notin K_1$.
    By the Mean Value Theorem, $\exists c \in (x,z_n)$ such that:
    \[
    (F^{\tau(z)})'(c) = [F^{\tau(z)} (z_n) - F^{\tau(z)}(z)]/[z_n - z] \geq |K_1|/|R_n|.
    \]
    Because $F^{\tau(z)}$ is convex, we obtain $G'(z) = (F^{\tau(z)})'(z) \geq (F^{\tau(z)})'(c) > |K_1|/|R_n|$.
    Recall $|R_n| \leq |K_1|/2$ by item \ref{item_GP_smooth} of Lemma \ref{lem_F_smooth}.
    Thus $G'(z) \geq 2$.
    Therefore $\forall x \in \Delta, G'(x) \geq \min\{m_2,2\}> 1$.
\end{proof}

\begin{lem}
    \label{lem_SSP_smooth}
    $F \in \mathcal{F}_*^{weak}$.
\end{lem}
\begin{proof}
    By the previous Lemma \ref{lem_GP_smooth}, we know $F \in \mathcal{F}$.
    By item \ref{item_image_smooth} of Lemma \ref{lem_F_smooth},
    the subintervals $L_n$ are consistent with the definition in (\ref{eqn_L_n}).
    By item \ref{item_ratio_smooth}, we can take $p_n = |L_n|/|K_n|$
    and satisfy Definition \ref{defi_F_star}, thus $F \in \mathcal{F}_*$.
    By item \ref{item_slope_smooth}, $F|_{L_n}$ is affine and thus convex, 
    thus $F \in \mathcal{F}_*^{weak}$.
\end{proof}

\begin{lem}
    \label{lem_smooth_flat}
    The pair $f_1, F$ glues flatly, and thus by Lemma \ref{lem_glueflat}, $f$ is $C^\infty$ at $q$.
\end{lem}
\begin{proof}
    Because $f_1,F$ are smooth and $F$ is adjusted to $f_1$, it is sufficient to verify each item of Definition \ref{defi_glueflat}. 
    Items \ref{item_flat_f1} and \ref{item_flat_m} are verified by choice of $f_1$ in Subsection \ref{subsection_f1_smooth}.
    Item \ref{item_flat_endpoints} is verified with $\alpha = 2$ by choice of $z_n-q \approx n^{-2}$ in Subsection \ref{subsection_smooth_partition}.
    Item \ref{item_flat_Q} is verified by item \ref{item_derivatives_smooth} of Lemma \ref{lem_F_smooth}.
    We must now verify item \ref{item_flat_P}.
    Fix a degree $k$ and an interval $K_n$. We want to find $\max_{x \in K_n} |(f_1^{-1})^{(k)}(F(x))|$.
    By the proof of Lemma \ref{lem_adjusted_gluing} we know $f_1^{-1} \circ F$ is affine for $x \in R_n$.
    So the maximum must be achieved on $x \in K_n \setminus R_n$.
    Because $f_1$ is concave at a neighborhood of $q$, the maximum must be achieved at $x = z_n - |R_n|$.
    By items \ref{item_slope_smooth}, \ref{item_Rn_smooth} and \ref{item_varphi_smooth} of Lemma \ref{lem_F_smooth} 
    we have $F(z_n - |R_n)) = \varphi(m_{n-1}|R_n|) \leq \varphi(Cn^{-4})$, for a constant $C > 0$.
    So by the bound (\ref{eqn_varphi_inv_bounds}), we obtain for a constant $C_k>0$ depending on $k$:
    \begin{align*}
    \max_{x \in K_n} |(f_1^{-1})^{(k)}(F(x))| 
    &\leq C_k \varphi(C n^{-4} )^{-k} 
    = C_k [\exp(-9 C^{-1/9} n^{4/9})]^{-k}\\
    &= C_k \exp(9k C^{-1/9} n^{4/9})
    \leq \exp(D_k n^{4/9}).
    \end{align*}
    For some $D_k > 0$.
    Which means item \ref{item_flat_P} is verified with $\beta_k = 4/9 < 1/2 = 1/\alpha$.
\end{proof}

\begin{lem}
    \label{lem_smooth_smooth}
    $f \in C^\infty(\mathbb{S}^1)$.
\end{lem}
\begin{proof}
    $f_1$ is $C^\infty$ on $(p,q)$ by construction.
    Because $F$ is adjusted to $f_1$, by Lemma \ref{lem_adjusted_gluing}, $f_2$ is $C^\infty$ on $(q,p)$.
    By Lemma \ref{lem_smooth_flat}, $f$ is $C^\infty$ at $q$.
    It remains to verify smoothness at $p$.
    By choice of $f_1$ and item \ref{item_flat_f1} of Lemma \ref{lem_F_smooth}, $f$ is affine with slope $m$
    on the interval $K_1 \cup [p,b)$, which is a neighborhood of $p$. 
    Thus $f$ is $C^\infty$ at $p=0=1$.
\end{proof}


\begin{proof}[Proof of Proposition \ref{prop_result_smooth}]
    We have chosen $f_1 : [0,q) \to [0,q)$ in Subsection \ref{subsection_f1_smooth}.
    and $F \in \mathcal{F}'$ in Lemma \ref{lem_F_smooth}.
    We obtain $f: \mathbb{S}^1 \to \mathbb{S}^1$ from Proposition \ref{prop_finite_realization}.
    Lemmas \ref{lem_GP_smooth}, \ref{lem_SSP_smooth} and \ref{lem_smooth_smooth} respectively give us $f \in \mathcal{D}$, $F \in \mathcal{F}_*^{weak}$, and $f \in C^\infty(\mathbb{S}^1)$.
    Note $f'(p) = m > 1$ as a consequence of the proof of Lemma \ref{lem_smooth_smooth}.
\end{proof}

\section{Construction for Proposition \ref{prop_result_full}} 
\label{section_construction_full}
In this section we prove Proposition \ref{prop_result_full}.
Our strategy is to choose a suitable pair $(f_1,F)$, then we obtain a unique $f$ by Proposition \ref{prop_finite_realization}.
First, fix a point $q \in (0,1/2)$, and define $\mathcal{P} := \{I_1 = [0,q), I_2 = [q,1)\}$.

\subsection{Choice of $f_1$}
\label{subsection_f1_full}
We choose $f_1:I_1 \to I$ to be a convex $C^\infty$ orientation preserving homeomorphism such that
$\lim_{x \to q^{-}} f'(x) = +\infty$, and there exists a point $b \in (0,q/2)$, 
such that the restriction $f_1|_{[0,b)}$ maps to $[0,q)$, 
and it is affine with slope $m := q/b > 2$.
Note that an explicit formula for $f_1$ can be written with standard algebraic functions such as $x \to -(q-x)^{1/2}$ and smoothing techniques, but it is not important to have such an explicit formula.

\subsection{Choice of Partition $\mathcal{P}^F$}
Before we choose $F$, we must choose a partition $\mathcal{P}^F= \{K_n\}_{n \geq 1}$.
This is done by choosing their endpoints $z_n$ (as in (\ref{eqn_endpoints})).
Let $a \in (0,1/2)$, then define:
\begin{align}
    \label{eqn_endpoint_full}
    z_1:= 1,\quad z_2 := 1-(1-q)/m, \quad
    \forall n \geq 3, z_n := q + (z_2 - q) a^{n-2}.
\end{align}
For future reference we observe their lengths:
\begin{align}
    \label{eqn_length_full}
    |K_1| &= z_2 - z_1 = (1-q)/m,\\
    \nonumber \forall n \geq 2, |K_n| &= z_n - z_{n+1} = (z_2 - q)(a^{n-2} -a^{n-1}).
\end{align}
Recall the union $K_n^- = \bigcup_{k \geq n+1} K_n$ defined in (\ref{eqn_K_n_minus}).
We note its length.
\begin{align}
    \label{eqn_union_length_full}
    |K_n^-| &= z_{n+1} - q = (z_2 - q)a^{n-1}.
\end{align}

\subsection{Choice of $F$}
\label{subsection_full_F}
The following Lemma \ref{lem_F_full} specifies a choice of $F \in \mathcal{F}'$ with several desired properties for further reference.

\begin{lem}
    \label{lem_F_full}
    There exists a map $F \in \mathcal{F}'$ with branch domains given by $\mathcal{P}^F$
    such that it is adjusted to $f_1$ (Definition \ref{defi_adjusted}), and furthermore $F$ together with the subintervals $L_n,R_n$ given in Definition \ref{defi_adjusted} satisfy:
    \begin{enumerate}[label= (\roman*)]
        \item \label{item_affine_full} The first branch $F|_{K_1}$ is affine with slope $m_1 = m>1$.
        \item \label{item_ratio_full} $\prod_{n \geq 1} |L_n| / |K_n| > 0$.
        \item \label{item_image_full} $\forall n \geq 1, F(L_n) = K_n^- = \bigcup_{k \geq n+1} K_n$
        \item \label{item_slope_full} For $n \geq 2$, $F|_{L_n}$ is affine with constant slope $m_n$ and $m_n \geq a/(1-a)>1$.
        \item \label{item_convex_full} Each branch $F|_{K_n}$ is convex.
    \end{enumerate}
\end{lem}
\begin{proof}
    To satisfy \ref{item_affine_full}, let $F|_{K_1} : K_1 \to (q,1)$ to be affine. 
    By (\ref{eqn_length_full}), it has constant slope $(1-q)/K_1 = m$.
    To satisfy item \ref{item_left_adjusted} of Definition \ref{defi_adjusted} and \ref{item_image_smooth} for $n = 1$, let $L_1$ be the subset of $K_1$ that maps to $K_1^-$.
    Choose a sequence $(p_n)_{n \geq 2}$ such that $\prod_{n \geq 2} p_n > 0$.
    For $n \geq 2$, we define the subinterval $L_n \subset K_n$ as follows:
    \begin{equation}
        \label{eqn_full_Ln}
        L_n := [z_{n+1}, z_{n+1} + p_n|K_n|), \quad |L_n| = p_n|K_n|.
    \end{equation}
    This choice of $L_n$ satisfies \ref{item_ratio_full}.
    To satisfy item \ref{item_left_adjusted} of Definition \ref{defi_adjusted} and \ref{item_image_full} for $n \geq 2$, 
    let $F|_{L_n} : L_n \to K_n^-$ to be affine.
    We compute by (\ref{eqn_length_full}), (\ref{eqn_union_length_full}), and (\ref{eqn_full_Ln}):
    \begin{equation}
        \label{eqn_full_mn}
        m_n = |K_n^-|/|L_n| = p_n^{-1}a/(1-a).
    \end{equation}
    For each $n \geq 2$, let $\delta_n > 0$ be a small number to be fixed later, then let $R_n = (z_n - \delta_n, z_n)$.
    To satisfy item \ref{item_right_adjusted} of Definition \ref{defi_adjusted}, we have
    $F|_{R_n}(x) = f( q + m_n(x-z_n))$.
    Recall by Subsection \ref{subsection_f1_full} that $f_1$ is convex and $\lim_{x \to q^-f'(x)} = \infty$.
    We now extend the graphs of $F|_{L_n}$ and $F|_{R_n}$ into $K_n \setminus (L_n \cup R_n)$, then connect the two graphs with a line segment to create a nonsmooth graph.
    If we assume $\delta_n$ to be small enough, then the graph is convex.
    We smoothen the graph by convolution with a smooth positive bump function.
    This gives us a smooth branch $F|_{K_n}$ which satisfies item \ref{item_smooth_adjusted} of Definition \ref{defi_adjusted}.
    Because convolution with positive kernels preserve convexity, we satisfy \ref{item_convex_full}.
\end{proof}

\begin{figure}[ht]
    \centering
    \includegraphics[width=0.4\linewidth]{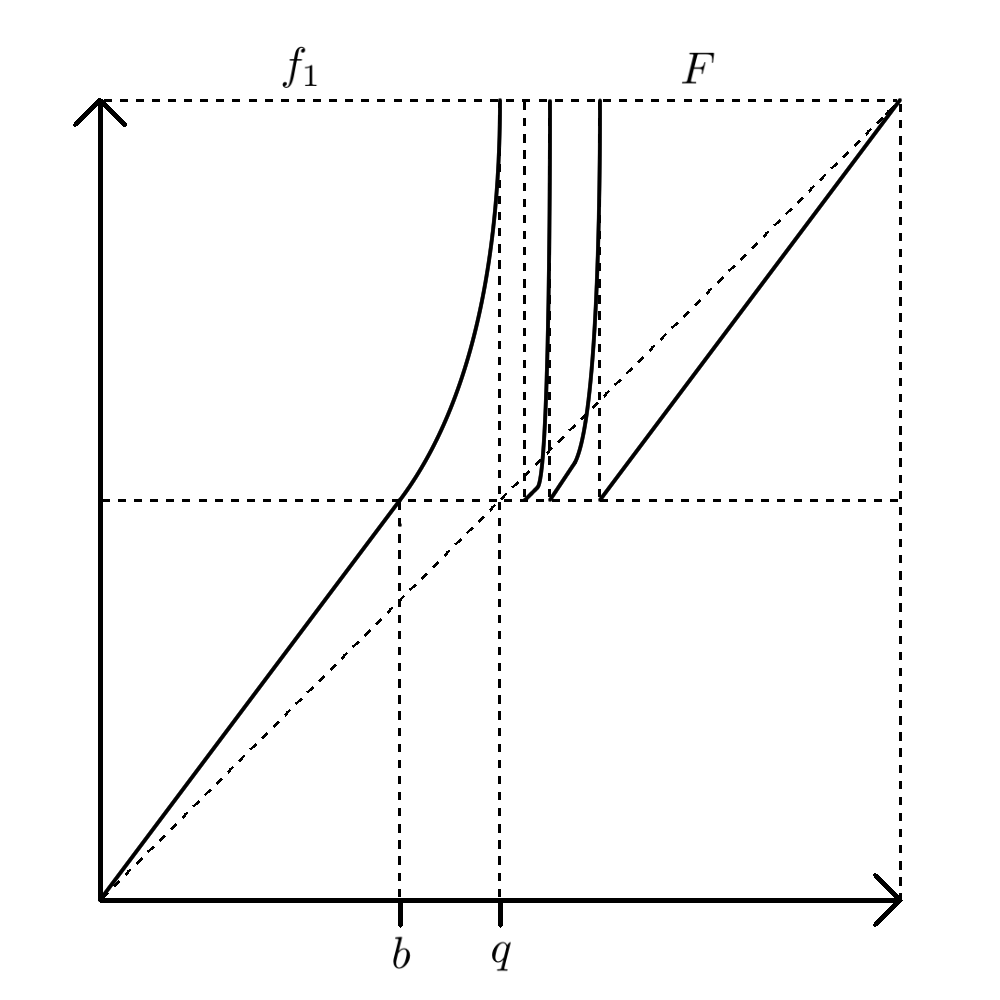}
    \caption{Sketch of $(f_1,F)$ with $q = 1/2, b = 3/8$.}
\end{figure}

\subsection{Verification of Properties}
We have chosen $f_1$ in Subsection \ref{subsection_f1_full}, 
and $F \in \mathcal{F}'$ in Lemma \ref{lem_F_full}.
By Proposition \ref{prop_finite_realization}, we obtain a full branch map $f: I \to I$ 
with branches defined on $\mathcal{P}=\{[0,q),[q,1)\}$.
To prove Proposition \ref{prop_result_full}, we must verify three properties:
$f \in \mathcal{D}$, $F \in \mathcal{F}_*^{strong}$, 
$f \in C^\infty(\mathbb{S}^1\setminus \{q\})$, and $f'(p) > 1$.
\begin{lem}
    \label{lem_GP_full}
    $F$ has a generating partition, thus
    $F \in \mathcal{F}$ and $f \in \mathcal{D}$.
\end{lem}
\begin{proof}
    By item \ref{item_affine_full}, \ref{item_slope_full} and \ref{item_convex_full}
    of Lemma \ref{lem_F_full}, for all $x \in (q,1)$, $F'(x) > > 1$.
    Thus $F$ is uniformly expanding, and has a generating partition.
\end{proof}

\begin{lem}
    \label{lem_SSP_full}
    $F \in \mathcal{F}_*^{strong}$.
\end{lem}
\begin{proof}
    By the previous Lemma \ref{lem_GP_full}, we know $F \in \mathcal{F}$.
    By item \ref{item_image_full} of Lemma \ref{lem_F_full},
    the subintervals $L_n$ are consistent with the definition in (\ref{eqn_L_n}).
    By item \ref{item_ratio_full}, we can take $p_n = |L_n|/|K_n|$
    and satisfy Definition \ref{defi_F_star}, thus $F \in \mathcal{F}_*$.
    By item \ref{item_convex_full}, each branch $F|_{K_n}$ is convex, thus
    $F \in \mathcal{F}_*^{strong}$.
\end{proof}

\begin{lem}
    \label{lem_smooth_full}
    $f \in C^\infty(\mathbb{S}^1 \setminus \{q\})$.
\end{lem}
\begin{proof}
    $f_1$ is $C^\infty$ on $(p,q)$ by construction.
    By Lemma~\ref{lem_adjusted_gluing}, because $F$ is adjusted to $f_1$, 
    we have $f_2$ is $C^\infty$ on $(q,p)$.
    We must check the regularity of $f$ at $p$.
    By item \ref{item_affine_full} of Lemma \ref{lem_F_full}
    $f$ is affine with slope $m$ on the interval $K_1 \cup [p,b)$, which is a neighborhood of $p$. 
    Therefore $f \in C^\infty(\mathbb{S}^1 \setminus \{q\})$.
\end{proof}

\begin{proof}[Proof of Proposition \ref{prop_result_full}]
    We have chosen $f_1 : [0,q) \to [0,q)$ in Subsection \ref{subsection_f1_smooth}.
    and $F \in \mathcal{F}'$ in Lemma \ref{lem_F_smooth}.
    We obtain $f: I \to I$ from Proposition \ref{prop_finite_realization}.
    Lemmas \ref{lem_GP_full}, \ref{lem_SSP_full} and \ref{lem_smooth_full} respectively give us $f \in \mathcal{D}$, $F \in \mathcal{F}_*^{strong}$, 
    and $f \in C^\infty(\mathbb{S}^1\setminus \{q\})$.
    Note $f'(p) = m > 1$ by item \ref{item_affine_full} of Lemma \ref{lem_F_full}.
\end{proof}

\section{Remarks on the Role of Flat Critical Points}
\label{section_flat}
In Section \ref{section_realization} we introduced
Lemma \ref{lem_glueflat}, which gives sufficient conditions on $(f_1,F)$ 
such that $f$ is smooth at $q$ with $q$ a flat critical point.
Because we used it to construct $f$ for Theorem \ref{thm_full}, $f$ must have a flat critical point $q$.
We now make brief informal remarks on the role of the flatness of $q$ in both our results,
and compare our strategy with results in the literature,
in particular the unimodal maps with flat critical point.

\subsection{Flatness and Regularity}
With our methods, $q$ being flat is necessary for regularity.
Assume $f \in C^\infty(\mathbb{S}^1)$ and $F \in \mathcal{F}_*$.
Because of how $\mathcal{F}_*$ is defined, $F'$ must be small in the subintervals $L_n \subset K_n$, and large outside of it.
As a consequence, $\max_{x,y \in K_n} f'(x)/f'(y) \to \infty$ as $n \to \infty$. 
Thus all derivatives of $q$ must vanish.

However, $q$ being flat is not strictly necessary to have a Sisyphus Attractor.
We can take the construction in Section \ref{section_construction_full}, and choose $a < b/q$.
Observe $z_n - q \approx a^n$, and $f(z_n) - f(q) \approx (b/q)^n$.
We obtain $\lim_{x \to q^+} [f(x)-f(q)]/[x-q] = \infty.$
Even though $f'(q)$ is not even defined, we still have a Sisyphus Attractor.

\subsection{A Difference of Strategy}
We now compare our Sisyphus Attractor to the unimodal maps with flat critical points shown in \cite{Zwe04}.
Two particular assumptions are needed to find a Sisyphus Attractor in that setting.
The first is nonpositive Schwarzian derivative, which implies some type of bounded distortion.
The second one is $\int \log|f'| = -\infty$, which is a nonintegrability condition
which in that setting is equivalent to having nonintegrable return time.

These assumptions are not met by the maps in our results.
Note in both our results, $F \in \mathcal{F}_*$. 
Thus $F'$ must vary dramatically in very small intervals, and so does $f'$.
This means unbounded distortion is actually the key for our results.

On the other hand, integrability of return time is independent from our results.
Note in (\ref{eqn_length_smooth}), we have $|K_n| \approx n^{-3}$, which gives an integrable return time.
One can also modify Section \ref{section_construction_smooth} to use $|K_n| \approx n^{-2}$, which gives a nonintegrable return time.





\bibliographystyle{unsrt}
\bibliography{ref}

\hfill \break

\end{document}